\documentclass[a4paper,reqno]{amsart}
\usepackage{times,fancyhdr}
\usepackage{amssymb}
\usepackage{amsmath}
\newtheorem{lemma}{Lemma}
\newtheorem{theorem}{Theorem}
\begin{document}
\setlength{\footskip}{1in}
\renewcommand{\headrulewidth}{0pt}
\vspace{\baselineskip}

\vspace{\baselineskip}

\begin{center}

\section*{Goldbach and Twin Prime Pairs: A Sieve Method to Connect the Two}

\vspace{\baselineskip}

\fontsize{8}{10}\selectfont 
Tom Milner-Gulland
\vspace{\baselineskip}

\fontsize{8pt}{10pt}\selectfont 
\end{center}
\vspace{\baselineskip}


\section*{Abstract}
This paper proposes, and demonstrates the efficacy of, a method for establishing a lower bound for cardinalities of selected sets of twin primes, and shows that the proof employed may be modified for selected sets of Goldbach pairs. Our sieve method is centred on the restrictive properties of intervals, specifically regarding divisibility distributions. We implicitly use the Chinese Remainder Theorem by way of the use of the midpoint in our intervals, and consider the sieve of Eratosthenes in such a way as to find a set of primes whose distribution is mirror-symmetrical about that midpoint. Bounds are established through the use of the formulae closely associated with the Prime Number theorem and the Mertens theorem. We show that the Goldbach conjecture is true if the Riemann hypothesis is true.
\vspace{\baselineskip}

Keywords: Goldbach pairs, twin primes, prime pairs, Goldbach and Twin Primes equivalence,
Euler totient, divisibility distributions, mirror symmetry, folded number scale.

\section*{Introduction}	

Throughout this paper, for any set $S$ that is explicitly stated to be an 'interval' or is written $[x,y]$ for some real $x$ and $y$, $S$ is to be taken to be a nonempty set of integers; $\mathbb{N}$ will be the set of non-negative integers; $\pi(x)$ will the number of primes not exceeding $x$; $p_{n}$ for $n = 1, 2,\ldots$ will be the sequence of primes; $\phi$ will be the Euler totient; for any finite subset $K$ of $\mathbb{N}$ and any integer $i$, $[K]^i$ will be the set of all subsets, $M$, of $K$ for which $|M|=i$. For any real $a$ and $b$, $(a, b)$ will denote an ordered pair unless written $(a, b) = 1$ or $(a, b) \ne 1$, which will denote that the greatest common denominator is one and not one respectively. 
Finally, $P(n)$ will be $\{p_{1}, p_{2},\ldots, p_{n}\}$.

For any integer $i$, we begin with the interval $[1, i]$ and any $i$-element interval, $I$. We study the number of elements of our relevant set, $T_I$, of integers (i.e. the union of $I$ and another set of integers, which serves to provide the conditions that we impose), that are not coprime to $\prod_{p \in J} p$, where $J$ is any set of primes. It is straightforward to show, as we do by combining Lemmas \ref{L1} and \ref{-AAJ-}, that $|T_{[1, |I|-1]}|-|T_{I}|$ is at most $|[J]^3|$. That is to say, in our scheme, each three-element subset of $J$ can serve to increase the number of elements of our relevant set, that are not coprime to $\prod_{p \in J}p$, by at most one. For the purposes of this introduction we may call any element of $[J]^3$ that yields in such a way, a \emph{yielding triple}. 
We show, through combining Lemmas \ref{L1} and \ref{L1A}, that for all four-element subsets, $W$, of $J$, it is impossible for $[W]^3$ to contain more than one yielding triple. 
This is the core of the proof of Theorem \ref{Th1}, which establishes an upper bound for the number of elements of $I$ that are not coprime to $\prod_{p \in J} p$. For any $k$, our upper bound on 
\begin{align}
	\notag |\{1+k \le m \le i+k: (m, \prod_{p \in J} p) \ne 1\}|-|\{1 \le m \le i: (m, \prod_{p \in J} p) \ne 1\}|
\end{align}
 is $5|J|^2/8$. As a matter of interest, we find by computer, through using the sieve of Eratosthenes combined with the Prime Number theorem that, for any $n$ and for $i=k =(p_n^2+1)/2$ and $J=P(n)$, the difference cited above, divided by $n^2$, converges to $\log(2)/4$, which is between one and two eighths.

Using $J=P(n)$, Theorem \ref{Th1} will ultimately enable us to use the Euler totient in conjunction, tacitly, with the Chinese Remainder Theorem. Our tacit use of the Chinese Remainder theorem is the part of the method that amounts to our essential concept of the folding of the number scale.  
We use the Mertens theorem to develop this approach.

\subsection{Extended introduction}

Our first theorem is a prelude to the sieve method that is the focus of this paper. For Theorems \ref{Th2} and \ref{Th3}, key to our method is the set $Z_{i,r}$, where $r$ is even and $i$ is as in our  Introduction (above), and which is specified as the set of all sets, $G$, for which we have the following. For each $s \in \{0,r\}$ and $p \in P(\pi(\sqrt{2i}))$, precisely one element of $\{\{ 1 \le m \le i: p \mid m-s\}: s \in \{0,r\}\}$ is in $G$. We ultimately use $Z_{i,r}$ in the form $\bigcup (\bigcup Z_{i,r})$. For each $K \in Z_{i,r}$, we show through Lemma \ref{-RIJ-} that $|\bigcup K|$ has an upper bound equal to that for the number of elements of any $i$-element interval that are not coprime to $\prod_{k=1}^n p_k$, which is found by Theorem \ref{Th1}. We then show how our sieve method may be used to address the Twin Primes conjecture when $r = 2$ and the Goldbach conjecture when $r = 2i$. Such is explained in the section Method Outline \ref{-MO-}. 

Let $n=\pi(\sqrt{2i})$. Consider any element, $E_p$, of 
\begin{align}
\notag \{\{ 1 \le m \le i: p \mid m-s\}: s \in \{0,r\}\},
\end{align}
where $p \in P(n)$, to be called a \emph{$p$-sieve} (this way, we employ the sieve of Eratosthenes, which justifies our use of $p_n^2$, as discussed below). Suppose we say that $E_p$ has a larger \emph{magnitude} than the $q$-sieve $E_q$, where $q \in P(n) \setminus \{p\}$, when $|E_p|>|E_q|$. Then we may formulate a value, $w_n$, using the upper bound, $j_n$, on $|\bigcup K|$ as found by Theorem \ref{Th1}. We treat $w_n$ in a similar way to the treatment of each $p \in P(n)$ with respect to the Euler totient. Specifically, we create a quasi-sieve, which might be called a $w_n$-sieve, that may be shown to provide an upper bound on $|\bigcup(\bigcup Z_{i,r})|$, given by
\begin{align}
 i\left(1-\frac{1}{2}\prod_{q \in (P(n)\setminus \{2\}) \cup \left\{\frac{i\prod_{k=1}^n\left(1-1/p_k\right)}{j_n-i\left(1-\prod_{k=1}^n\left(1-1/p_k\right)\right)}\right\}}\left(1-\frac{2}{q}\right)\right). \label{-ZIR-}
\end{align}
(The numerator two, for $2/q$, where $q$ is the bound variable cited as being any element of $P(n)\setminus \{2\}$, is attributable to the folding of the number scale.) In the above expression, $w_n$ is the element in the set, beneath the product, for which there is a union with $P(n) \setminus \{2\}$.
Such a use of $w_n$ is justified by our forthcoming \eqref{-MIU-} (take it that $T$, in \eqref{-MIU-}, is any element of $Z_{i,r}$ and  $K_T$ is any subset of $T$; take it also that $u_{I,n}$ as in \eqref{-MIU-} is equal to $w_n$), combined with the fact that the second term between the outer brackets on the right side of the first relation of \eqref{-MIU-} is equal to zero when $K_T=T$.

 In an analogous way, we can now apply our idea of sieve magnitude to the quasi-sieve; for the purposes of this introduction we can use the phrase \emph{quasi-sieve magnitude}. Indeed, from here, Theorems \ref{Th2} and \ref{Th3} follow by simple algebra combined with known bounds for the prime counting function, all combined, in the case of the Twin Primes conjecture where we take $r=2$, with the Mertens theorem. To address the Goldbach conjecture, we use a work by Nicolas, through which a lower bound on the prime count, connected with the Euler-Mascheroni constant, can be deduced subject to the Riemann hypothesis being true. 

\section{Furnishing a Sieve Method}

\begin{theorem} \label{Th1}
Let $J$ be any set of primes for which $|J| \ge 4$. Let $I$ be any interval. Then
\begin{align}
\notag  |\{m \in I: (m, \prod_{p \in J}p) \ne 1\}|  -|\{m \in [0, |I|-1]: & (m, \prod_{p \in J}p) \ne 1\}|\\
\le&  \frac{5|J|^2}{8}.\label{-SVSE}
\end{align}
\end{theorem}
\vspace{\baselineskip}

\begin{lemma} \label{-BI-}
Let $A$ and $A'$ be any set of integers and $T$ be any set of primes.
For any $V \in \{A, A'\}$, let 
\begin{align}
	\notag & h_{V,T} = \sum_{m \in V}  |\{p \in T: p \mid m\}| - \sum_{m \in V} |[\{p \in T: p \mid m\}]^2|\\
 &\ \ \ \ \ \ \  \  \ \ \ \ +  \sum_{\substack{m \in V\\ |\{p \in T: p \mid m\}|>2}} (|[\{p \in T: p \mid m\}]^2| - |\{p \in T: p \mid m\}|  +1). \label{-FCR-}
\end{align} 
Then
\begin{align}
|\{m \in A: (m, \prod_{p \in T} p) \ne 1\}| - |\{m \in A': (m, \prod_{p \in T} p) \ne 1\}|  = h_{A,T}-h_{A',T}.
   \label{-BKG-}  
\end{align}
\end{lemma}

\begin{proof}
We have
\begin{align}
\notag |\{m \in V: & (m, \prod_{p \in T}p) \ne 1\}| \\
\notag =&  \sum_{\substack{m \in A}} |\{p \in T: p \mid m\}| - \sum_{\substack{m \in V\\ (m, \prod_{p \in T}p) \ne 1}} (|\{p \in T: p \mid m\}|-1) \\
\notag = & \sum_{m \in V} |\{p \in T: p \mid m\}| - \sum_{m \in V} |[\{p \in T: p \mid m\}]^2|\\
&\ \ \ \ \ \ \ \  + \sum_{\substack{m \in V\\ |\{p \in T: p \mid m\}|>2}} (|[\{p \in T: p \mid m\}]^2| -  |\{p \in T: p \mid m\}|+1), 
   \label{-BIG-}  
\end{align}
implying \eqref{-BKG-}. Here, the 'minus one' and, as in the case where the immediately preceding term is negative, 'one' terms in the expressions for the summands are found by the following. There is necessarily at least one element of $T$ that divides the bound variable $m$ cited in the $|\{p \in T: p \mid m\}|-1$ that is the expression for the summands of the second term on the right side of the first relation of \eqref{-BIG-}. Accordingly, $|\{m\} \cap \{s \in V: (s, \prod_{p \in T}p) \ne 1\}|=1$ is a constant function of $|\{p \in T: p \mid m\}|$.

 Further, for all $m$ for which $|\{p \in T: p \mid m\}|=2$, we have 
\begin{align}
	|\{m\}| = |\{p \in T: p \mid m\}| - |[\{p \in T: p \mid m\}]^2|,
\end{align}  
 giving the term, in the second relation, $- \sum_{m \in V} |[\{p \in T: p \mid m\}]^2|$. Contrastingly, for all $k$ for which $|\{p \in T: p \mid k\}|>2$ we have 
\begin{align}
	|\{k\}| \ne |\{p \in T: p \mid k\}| - |[\{p \in T: p \mid k\}]^2|
\end{align} 
which, through the right side of the first relation, accounts for the final two terms of the final expression for the summands and completes the proof.
\end{proof}

\subsection{Definition.} \label{D.E}  
Let $I$ be as in Theorem \ref{Th1}. For any $t \in \{0,1\}$, 
 any set $L$ of primes and any integer $r$, let $Q_{I,L,t,r}$ and $Q'_{I,L,t,r}$ be any sets of integers for which I to III, below, are all true:

\vspace{\baselineskip}

\noindent I.  $\max\{ |\{p \in L: p \mid m\}|: m \in Q_{I,L,t,r} \cup Q'_{I,L,t,r}\} \le r$;
\vspace{\baselineskip}

\noindent II. for each $s \in \{j \in \mathbb{N}: 1 \le j \le r\}$ and each $M \in [L]^s$ 
\begin{align} 
 |\{m \in [t, |I|-1] \cup Q_{I,L,t,r}: \prod_{p \in M}p \mid m\}| =	 |\{m \in I \cup Q'_{I,L,t,r}: \prod_{p \in M}p \mid m\}|; \label{-RGO-}
\end{align} 
\vspace{\baselineskip}

\noindent III. $|[t, |I|-1] \cup Q_{I,L,t,r}| = |I \cup Q'_{I,L,t,r}|$.
\vspace{\baselineskip}

We note that \eqref{-RGO-} implies that
\begin{align} 
	\sum_{m \in [t, |I|-1] \cup Q_{I,L,t,r}} |[\{p \in L: p \mid m\}]^s| =	\sum_{m \in I \cup Q'_{I,L,t,r}}  |[\{p \in L: p \mid m\}]^s|.
\end{align} 

\begin{lemma} \label{L1}
Let $I$ be as in Theorem \ref{Th1}. Then for any three-element set $H$ of primes for which 
\begin{align}
\notag   |\{m \in I: \prod_{p \in H}p \mid m\}| - |\{1 \le m \le |I|-1: \prod_{p \in H}p \mid m\}| = 1
\end{align}
we have
\begin{align}
\notag & |\{m \in I \cup Q'_{I,H,1,2}: (m, \prod_{p \in H}p) \ne 1 \}|  - |\{m \in  [1, |I|-1] \cup Q_{I,H,1,2}: (m, \prod_{p \in H}p) \ne 1\}|\\
\notag &\  \  \  \  \ \  \  \ \ \ \ \ \ \ \ \ \ \ \ \ \ \  \ \ \ \ \ = 1.
\label{-BHB-}  
\end{align}
\end{lemma}

\begin{proof}
Conditions I to III in Definition \ref{D.E} require through Lemma \ref{-BI-} for $A =  I \cup Q'_{I,H,1,2}$ and $A'=[1, |I|-1] \cup Q_{I,H,1,2}$ and $T=H$, that
\begin{align} 
	\notag & |\{m \in I \cup Q'_{I,H,1,2}: (m, \prod_{p \in H} p) \ne 1\}| -|\{m \in [1, |I|-1] \cup Q_{I,H,1,2}: (m, \prod_{p \in H} p) \ne 1\}|\\
	\notag & \  \ \ \ \ =  \sum_{\substack{m \in  I \cup Q'_{I,H,1,2}\\ |\{p \in H: p \mid m\}|>2}} (|[\{p \in H: p \mid m\}]^2| - |\{p \in H: p \mid m\}| +1) \\
	\notag &\ \ \ \ \ \ \ \ \  \   - \sum_{\substack{m \in [1, |I|-1] \cup Q_{I,H,1,2} \\ |\{p \in H: p \mid m\}|>2}} (|[\{p \in H: p  \mid m\}]^2| - |\{p \in H: p \mid m\}| +1)\\ 
		& \  \ \ \ \  = 1
\end{align}
which completes the proof. 
\end{proof}

\begin{lemma} \label{L1A}
 Let $I$ be as in Theorem \ref{Th1}. Let $Z$ be any four-element set of primes for which
\begin{align}
 |\{0 \le m \le |I|-1: \prod_{p \in Z}p \mid m\}|	- |\{m \in I: \prod_{p \in Z}p \mid m\}|  = 1.
\end{align} 
Then 
\begin{align}
\notag & |\{m \in I \cup Q'_{I,Z,0,3}: (m, \prod_{p \in Z}p) \ne 1 \}|  - |\{m \in  [0, |I|-1] \cup Q_{I,Z,0,3}: (m, \prod_{p \in Z}p) \ne 1\}|\\
&\  \  \  \  \ \  \  \ \ \ \ \ \ \ \ \ \ \ \ \ \ \  \ \ \ \ \ = 1. \label{-BIB-}  
\end{align} 
\end{lemma}
 
\begin{proof} 
By Lemma \ref{-BI-}, for $A = I \cup Q'_{I,Z, 0,3}$, $A'=[0, |I|-1] \cup Q_{I, Z, 0, 3}$ and $T=Z$, the conditions on $Q_{I,Z,0,3}$ and $Q'_{I,Z,0,3}$ require that, for some $0 \le m <|I|-1$ such that $|\{p \in Z: p \mid m\}|=4$,
\begin{align}
\notag & |\{u \in I \cup Q'_{I,Z,0,3}: (u, \prod_{p \in Z}p) \ne 1\}|- |\{u \in [0, |I|-1] \cup Q_{I,Z,0,3}: (u, \prod_{p \in Z}{p}) \ne 1\}|\\
\notag &\ \  \  = \sum_{\substack{u \in I \cup Q'_{I,Z,0,3} \\ |\{p \in Z: p \mid u\}|>2}} (|[\{p \in Z: p \mid u\}]^2| - |\{p \in Z: p \mid u\}| +1)\\
\notag &\ \ \ \ \ \ \ - \sum_{\substack{u \in [0, |I|-1] \cup Q_{Z,J,0,3}\\ |\{p \in Z: p \mid u\}|>2}} (|[\{p \in Z: p \mid u\}]^2|\\
\notag &\ \ \ \ \ \ \ \ \ \ \ \ \ \ \ \ \ \ \ \ - |\{p \in Z: p \mid u\}| +1)\\
\notag &\ \  \  = |[\{p \in Z: p \mid m\}]^3|\left({3 \choose 2}-3+1\right)-(|[\{p \in Z: p \mid m\}]^2|\\
\notag &\ \ \ \  \ \ \ \ \ \ \ \ \ -|\{p \in Z: p \mid m\}|+1)\\
\notag &\ \  \  = 4-(6-4+1)\\  
 &\ \  \  = 1
\end{align} 
which completes the proof.
\end{proof}

\begin{lemma} \label{-AAJ-}
Let $J$ be as in Theorem \ref{Th1}. Let $A$ and $A'$ be any sets of integers for which, for each $r \in \{1,2,3\}$,
\begin{align}
\sum_{m \in A} |[\{p \in J: p \mid m\}]^r| = \sum_{m \in A'} |[\{p \in J: p \mid m\}]^r| \label{-PJR-}
\end{align}
and 
\begin{align}
\notag 	& \sum_{\substack{m \in A \\ |\{p \in J: p \mid m\}| >3}} \frac{|[\{p \in J: p \mid m\}]^3|}{|\{u \in A: |\{p \in J: p \mid u\}| >3\}|} \\
\notag	&\ \  \ \   \ \ \ \  \ \ \  \  > \sum_{\substack{m \in A' \\ |\{p \in J: p \mid m\}| >3}} \frac{|[\{p \in J: p \mid m\}]^3|}{|\{u \in A': |\{p \in J: p \mid u\}| >3\}|}. 
\end{align} 
Then
\begin{align}
	|\{m \in A: (m, \prod_{p \in J}p) \ne 1\}| < |\{m \in A': (m, \prod_{p \in J}p) \ne 1\}|. \label{-BHW-}
\end{align} 
\end{lemma}

\begin{proof}
We have I and II, below.
\vspace{\baselineskip}

\noindent I. For all $k \geq 3$,
\begin{align}
\notag \frac{{k   \choose  3}k}{{k   \choose 2}^2} =&\ \frac{4(k -1)(k -2)}{6(k -1)^2}\\
=& \frac{2(k-2)}{3(k-1)} \label{-UTY}
\end{align}
is an increasing function of $k$. We may substitute $k = |\{p \in J: p \mid m\}|$ for any $m \in [1, |I|-1] \cup I$ such that $|\{p \in J: p \mid m\}| \ge 3$. 
\vspace{\baselineskip}

\noindent II. In \eqref{-UTY}, for $k = |\{p \in J: p \mid m\}|$, the quotients express
\begin{align}
\frac{{|\{p \in J : p \mid m\}| \choose 3 }}{ {|\{p \in J : p \mid m\}| \choose 2}}
\end{align}
as a ratio of 
\begin{align}
	\frac{{|\{p \in J : p \mid m\}| \choose 2 }}{ |\{p \in J : p \mid m\}|}.
 \end{align}  
\vspace{\baselineskip}

Combining I and II gives
\begin{align}
 \sum_{\substack{m \in A \\ |\{p \in J: p \mid m\}| \in \{1,2\}}} \frac{|[\{p \in J: p \mid m\}]^2|}{ |\{p \in J: p \mid m\}|}   > \sum_{\substack{m \in A' \\ |\{p \in J: p \mid m\}| \in \{1,2\}}} \frac{|[\{p \in J: p \mid m\}]^2|}{ |\{p \in J: p \mid m\}|}. \label{-APJ-}
\end{align} 
Further, 
\vspace{\baselineskip}

\noindent III. \eqref{-PJR-} for $r=1$ implies that
\begin{align}
\notag & |\{m \in A': (m, \prod_{p \in J}p) \ne 1\}| -	|\{m \in A: (m, \prod_{p \in J}p) \ne 1\}| \\
&\ \  \ \ \ \ \ \ \ \  = \sum_{m \in A } (|\{p \in J: p \mid m\}|-1) - \sum_{m \in A'} (|\{p \in J: p \mid m\}|-1)
\end{align}  
with, for any $m \in A \cup A'$, $|\{p \in J: p \mid m\}|-1 = 0$ when $|\{p \in J: p \mid m\}|=1$. 
\vspace{\baselineskip}

\noindent For all $m \in A \cup A'$ for which $|\{p \in J: p \mid k\}| > 2$, and all $k$ for which $|\{p \in J: p \mid k\}| = 2$ we have
\begin{align}
	\frac{|[\{p \in J: p \mid m\}]^2|}{|\{p \in J: p \mid m\}|} > \frac{|[\{p \in J: p \mid k\}]^2|}{|\{p \in J: p \mid k\}|}.
\end{align}  
Therefore, combining III and \eqref{-APJ-} gives \eqref{-BHW-}.
\end{proof}

\begin{lemma} \label{-IQM-}
Let $I$ and $J$ be as in Theorem \ref{Th1}. Then
\begin{align}
\notag & |\{m \in I \cup Q'_{I,J,0,2}: (m, \prod_{p \in J}p) \ne 1 \}|  - |\{m \in  [0, |I|-1] \cup Q_{I,J,0,2}: (m, \prod_{p \in J}p) \ne 1\}|\\
&\  \  \  \  \ \  \  \ \ \ \ \ \ \ \ \ \le \frac{|J|^2}{8}. \label{-QIJ-}  
\end{align}
\end{lemma}

\begin{proof}
Let $S_J$ be any subset of $[J]^3$ for which $|S_J|$ is an upper bound for the value on the left side of \eqref{-QIJ-}. We shall proceed progressively to justify choosing $S_J$ so that $|S_J|$ is equal to the right side of \eqref{-QIJ-}.
	
	For any distinct $H$ and $H'$ in $[J]^3$, the fact that zero divides $\prod_{p \in H \cup H'} p$ implies that
\begin{align}
	\notag & |\{1 \le m \le |I|: \prod_{p \in H}p \mid m\} \cap \{1 \le m \le |I|: \prod_{p \in H'}p \mid m\}|\\
	&\ \ \ \le |\{m \in I: \prod_{p \in H}p \mid m\} \cap \{m \in I: \prod_{p \in H'}p \mid m\}|. \label{-MIP-}
\end{align}
	Further, for each $M \in [J]^3$ we have 
\begin{align}
	|\{0 \le m \le |I|-1: \prod_{p \in M}p \mid m\}| \ge |\{m \in I: \prod_{p \in M}p \mid m\}|.
\end{align} 
Therefore, combining all of Lemmas \ref{L1}, \ref{L1A} and \ref{-AAJ-} for $A= [0, |I|-1] \cup Q_{I,J,0,3}$ and $A'= I \cup Q'_{I,J,0,3}$, gives, for all subsets, $N$, of $J$,
\begin{align}
 \notag & |\{m \in I \cup Q'_{I,J,0,3}: (m, \prod_{p \in J}p) \ne 1 \}|  - |\{m \in  [0, |I|-1] \cup Q_{I,J,0,3}: (m, \prod_{p \in N}p) \ne 1\}| \\
\notag &\ \ \  \ \ \ \ \ \ \ \ \  \ \  \  \le  |\{m \in I \cup Q'_{I,J,1,2}: (m, \prod_{p \in N}p) \ne 1 \}|\\
\notag &\ \ \  \ \ \ \ \ \ \ \ \  \ \  \ \ \ \ \ \ \ \  \ \ \   - |\{m \in  [1, |I|-1] \cup Q_{I,J,1,2}: (m, \prod_{p \in N}p) \ne 1\}| + 1\\
&\ \ \  \ \ \ \ \ \ \ \ \  \ \  \  \le |[N]^3|+1. \label{-PJP-}
\end{align} 
Here, the first relation, when taken together with the second, is found by combining Lemma \ref{L1} and Lemma \ref{-AAJ-}, itself combined with \eqref{-MIP-}. The first term of the left side of the second relation is found by Lemma \ref{L1}. The second term is found by Lemma \ref{L1A} for its use of zero for the third parameter of $Q_{I,Z,0,3}$ and $Q'_{I,Z,0,3}$, where $Z$ is as in Lemma \ref{L1A} with $Z \in [J]^4$. We note that, in \eqref{-PJP-}, for each $k \in \{0,1\}$ and $d \in \{2,3\}$, when the second parameter of $Q_{I,J,k,d}$ and $Q'_{I,J,k,d}$ is given as $N$ instead of $J$, the result is unchanged. The second term on the right side of \eqref{-PJP-} is found by the fact that the left-hand endpoint, zero, of $[0, |I|-1]$ is not coprime to $\prod_{p \in N}p$.

	  For any $M \in [J]^3$ and $W \in [J]^4$ for which $M \subset W$, let $f_{S_J,M,W}$ be equal to one when $M \in S_J$ and equal to zero when $M \notin S_J$.
	
	When we vary $N$, we see the following. It follows through \eqref{-PJP-} that, for all $W \in [J]^4$ and any $D_W \in [W]^3$, the combination of Lemma \ref{L1} for $H=D_W$, and Lemma \ref{L1A} for $Z= W$ justifies our choosing $S_J$ so that $|[W]^3|-1=3$ elements of $[W]^3$ are not in $S_J$. This is to say, the fact that the right sides of both \eqref{-BHB-} and \eqref{-BIB-} are equal to one implies I, below. 
	\vspace{\baselineskip}
	
	\noindent I. For all $W \in [J]^4$ for which $f_{S_J,D_W,W}=1$, for each $M \in [W]^3 \setminus \{D_W\}$ our assumptions on $S_J$ allow that $f_{S_J,M,W}=0$. 
	\vspace{\baselineskip}
	
	For any $H \in [J]^3$, any $F \in [J]^4$ for which $H \subset F$, and any $T \in [H]^2$, let $g_{S_J,T,H,F} = f_{S_J,H,F}$. Then for any $T'$ in $[J]^2 \setminus \{T\}$, if $ g_{S_J,T,H,F} = 1$ and any $V \in [J]^4$ and $U \in [V]^3$ for which $T' \subset U$, it follows by I that II, below, is true.
	\vspace{\baselineskip}
	
	\noindent II. We have $ g_{S_J,T',U,V} = 1$ only if $T \cap T' = \emptyset$. Otherwise, when $T \cup T' \ne H$ while $T \cap T' \ne \emptyset$, for some distinct $L$ and $L'$ in $[T \cup T' \cup H]^3$ we would have $f_{S_J,L, T \cup T' \cup H} = 1$ and $f_{S_J,L', T \cup T' \cup H}=1$, which is contrary to I, above.
	\vspace{\baselineskip}
	
	\noindent For any subset $K$ of $[J]^2$ for which $|K|=\lfloor|J|/2\rfloor$, combining \eqref{-PJP-} and II (above) gives
	\begin{align}
 \notag  	& |\{m \in I \cup Q'_{I,J,0,2}: (m, \prod_{p \in J}p) \ne 1 \}|  - |\{m \in  [0, |I|-1] \cup Q_{I,J,0,2}: (m, \prod_{p \in J}p) \ne 1\}| \\
	\notag & \ \ \  \ \  \le     {\lfloor|K|/2\rfloor \choose 2} +1\\
  & \ \ \  \ \  < \frac{|J|^2}{8}
\end{align}  
where the final two parameters of $Q_{I,J,0,2}$ and $Q'_{I,J,0,2}$ are justified through the final relation of \eqref{-PJP-} (where we use $Q_{I,J,1,2}$ and $Q'_{I,J,1,2}$), combined with the fact that we have the second term on the right side of \eqref{-PJP-}. We thereby have \eqref{-QIJ-}.
\end{proof}

\subsection{Proof of Theorem \ref{Th1}}

\begin{proof}
Let $A$ and $A'$ be as in Lemma \ref{-BI-}. Let $I$ and $J$ be as in Theorem \ref{Th1}. Then, when $A$ and $A'$ are each intervals with $|A|= |A'|$,
\begin{align}
	\notag & \left|\sum_{m \in A} |[\{p \in J: p \mid m\}]^2| -\sum_{m \in A'} |[\{p \in J: p \mid m\}]^2| \right|\\
	\notag &\ \ \ \  \ \ \  \ \ \ \  \le |[J]^2|\\
 &\ \ \ \  \ \ \  \ \ \ \  =   \frac{|J|(|J|-1)}{2}. 
\end{align}
Therefore, combining Lemma \ref{-IQM-} and Lemma \ref{-BI-} for $A= I$ and $A'=[0,|I|-1]$ gives
\begin{align}
\notag  & |\{m \in I: (m, \prod_{p \in J}p) \ne 1\}| -|\{m \in [0, |I|-1]: (m, \prod_{p \in J}p) \ne 1\}|\\
\notag &\ \ \ \  \ \ \ \ \ \ \  \ \ \  \ \ \ \ \  \ \ \  \le  \frac{|J|^2}{8} + \frac{|J|(|J|-1)}{2}\\
\notag &\ \ \ \  \ \ \ \ \ \ \  \ \ \  \ \ \ \ \  \ \ \  < \frac{|J|^2}{8} + \frac{|J|^2}{2}\\
\notag &\ \ \ \  \ \ \ \ \ \ \  \ \ \  \ \ \ \ \  \ \ \  = \frac{(2+8)|J|^2}{16} \\
 &\ \ \ \  \ \ \ \ \ \ \  \ \ \  \ \ \ \ \  \ \ \  = \frac{5|J|^2}{8}.  \label{-TGL-}
\end{align}
Since for each $p \in J$, 
\begin{align}
\notag 	|\{m \in I: p \mid m\}| \le |\{0 \le m \le |I|-1: p \mid m\}|,
\end{align}
in \eqref{-TGL-} we have tacitly substituted zero for the two that is the final parameter of $Q_{I,J,0,2}$ and $Q'_{I,J,0,2}$. Condition I in Definition \ref{D.E} enables us to dispense with $Q_{I,J,0,0}$ and $Q'_{I,J,0,0}$, whence \eqref{-SVSE} follows.
\end{proof}

\section{The Folded the Number Scale}

\subsection{Method outline.} \label{-MO-}
\vspace{\baselineskip}

 Let $y$ be any integer $>1$. Our forthcoming exposition employs mirror symmetry in the context of the interval, $[1,y]$ and, for any integer $n$, the primes in $P(n)$ that divide $y$ and finally the sieve of Eratosthenes, which in turn contextualises our 'fold' of the number scale. For the Goldbach conjecture our interest is in primes $a$ and $b$ in $[1,2y]$ for which $a + b = 2y$. Here, one side, which we may call the lower side, of the fold will be taken to be $[1, y]$; the other, the upper side, will be $[y, 2y]$. Indeed, the context of our use of $y$ will imply a rephrasing of the Goldbach Conjecture, familiar as \emph{every even number greater than two is the sum of two primes}, to the equivalent \emph{every integer greater than three is the arithmetic mean of two primes}. 

For $i$ and $r$ as in our Extended Introduction, consider the set we denoted by $Z_{i,r}$. For both the Goldbach and Twin primes conjectures, the fold of the number scale occurs at $r/2$. When we address the Goldbach conjecture we use $r=2y$. For the Twin Primes conjecture we instead take $r=2$ and $y=p_n^2$. Thus, in contrast to our treatment of the Goldbach conjecture, in which we consider only positive integers, by folding the number scale we tacitly map contiguous negative integers, descending from $-1$, onto positive integers up to $p_n^2$. Then, using the sieve of Eratosthenes to establish coprimality, when we reverse the signage of each of the negative integers the value $|\bigcup Z_{p_n^2,2}|$ will be seen to be the number of primes, $a$, such that $p_n+2< a < p_n^2$, for which $a-2$ is also prime. 

 In short, we shall consider the objects of interest in our folded number scale as being subdivided into two distributions. One such is the distribution of all integers that are coprime to $\prod_{k = 1}^n{p_k}$, where $n$ is any integer. This distribution is the object that is folded, which is to say that when we apply functions to it, they are in the context of the mirror symmetry discussed above. The other is, in effect, an intermediate to the two sides of this folded distribution and one that tacitly invokes the Chinese Remainder theorem.

\subsection{Remark} \label{-r-}
Let $n\ge 1 $ and 
 $I$ be as in Theorem \ref{Th1}. Key to our method is the expression
\begin{align}
\prod_{k=1}^{n}\left(1-\frac{1}{p_k}\right)\prod_{k=2}^{n}\left(1-\frac{1}{p_k-1}\right).	 \label{-HYH-} 
\end{align}
We note that, for any $p \in P(n)$ and any $r$, when $p \mid r$ we have (by virtue of the mirror symmetry, about $r/2$), of the distribution of integer multiples of $p$), 
\begin{align}
 \{m \in I: p \mid m \}= \{m \in I: p \mid m-r\}, \label{YE-}
\end{align} 
but when $p \nmid r$, \eqref{YE-} does not hold, bringing into play the expression
\begin{align}
\prod_{k=1}^n\left(1-\frac{1}{p_k}\right)\prod_{\substack{q \in P(n)\\ q \nmid r}}\left(1-\frac{1}{q-1}\right).	\label{-GHJ-}
\end{align}
Hence the fact that, when $p=2$, 
\begin{align}
	\left(1-\frac{1}{p}\right)\left(1-\frac{1}{p-1}\right) = 0,
\end{align}	
ultimately justifies our imposing the condition in our forthcoming exposition that $r$ is even.

\subsection{Introduction to Theorems \ref{Th2} and \ref{Th3}.} 

For $i$ and $r$ as in our Extended Introduction, recall that $Z_{i,r}$ is the set of all sets, $G$, for which, for each $s \in \{0,r\}$ and $p \in P(\pi(\sqrt{2i}))$, precisely one element of $\{\{ 1 \le m \le i: p \mid m-s\}: s \in \{0,r\}\}$ is in $G$. Recall further that we noted that we shall ultimately use $Z_{i,r}$ in the form $\bigcup (\bigcup Z_{i,r})$. For any integer $n$, we shall show that the proof of the Goldbach conjecture resides fundamentally in taking $r=2z_n$ where $p_n^2/2 < z_n < p_{n+1}^2/2$ and $i=z_n$, and the proof of the Twin Primes conjecture, on which we shall focus in this introduction (since the justification of Theorem \ref{Th3} is essentially founded in the lemmas that prove Lemma \ref{Th2}), taking $r=2$ and $i=p_n^2$. Theorem \ref{Th1} enables us to use \eqref{-ZIR-} as an upper bound for $|\bigcup (\bigcup Z_{i,r})|$.

In what follows, first we use the Euler totient. Second, we implicitly use, multiplicatively, the value $ \prod_{p \in \{P(n): p \nmid r\}} (p-2)$, where $n$ is any integer and $r$ is even. This will enable future use, in the form of \eqref{-HYH-}, of our resulting expression (by substituting $P(n) \setminus \{2\}$ for the implicitly used $\{p \in P(n): p \nmid r \}$ and using our forthcoming \eqref{-HHH}). We note, incidentally, that since $\phi(p)-1=p-2$ while $\phi(p) = p-1$, \eqref{-GHJ-} is equal to $(1/2)\prod_{k=2}^{n}(1-2/p_k)$ when no prime in $P(n) \setminus \{2\}$ divides $r$.

 Key to our use of the bound variable $m$, for $\{m \in I : (m, \prod_{k=1}^n p_k) \ne 1\}$ and for $\{m \in I : (m-r, \prod_{k=1}^n p_k) \ne 1\}$, where $I$ is any interval, is the value $m(m-r)$. Here, the primes in $P(n)$ that divide either $m$ or $m-r$ are also those primes in $P(n)$ that divide $m(m-r)$. Substituting $r=2$ will give our proof of the Twin Primes conjecture. Here,
 \begin{align}
	 \left\{(m, m-2): p_n+2 \le m \le \frac{p_n^2+1}{2}\ \&\ \left(m(m-2), \prod_{k=1}^n p_k\right) = 1 \right\} 
 \end{align}
 is a subset of the set of all $(p, p-2)$ such that $p < p_n^2/2$, and $p$ and $p-2$ are together twin primes.

 \begin{theorem}\label{Th2} 
There are infinitely many pairs, $(p, q)$, of primes such that $p+2=q$.
\end{theorem}

\begin{lemma} \label{-BN-} 
Let $s$ be even. Let $J$ be any set of integers for which, for any $z \in J$, $z$ is coprime to $\prod_{d \in J \setminus \{z\}} d$. Then  
\begin{align}
\left|\left\{ 1 \le m \le \prod_{d \in J}d  : \left(m(m-s), \prod_{d \in J} d \right) = 1 \right\}\right|
=\ \phi(\prod_{d \in  J}d) \prod_{\substack{d \in J\\
 d \nmid s }}\left(1-\frac{1}{d-1}\right). \label{-HHH} 
\end{align}
\end{lemma}

\begin{proof} 
For each $q \in J$ for which $q \nmid s$ and all $i \in \mathbb{N}$, 
\begin{align}
|\{1+i \le m \le q+i: (m(m-s), q) = 1\}|=\ |\{1+i \le m \le q+i: (m, q) = 1 \}|-1. \label{-XX-}
\end{align}
 Since the first term on the right side of \eqref{-XX-} is equal to $\phi(q) =q-1$ and the left side is equal to $\phi(q)-1 =q-2$, and since for any real $x$, $1-1/x = (x-1)/x$, taking $x= q-1$ for our final line we have
\begin{align} 
\notag 
 |\{ 1 \le m \le \prod_{d \in J}d  : (m(m-s), \prod_{d \in J}d)  = 1\}|  =&\  \phi(\prod_{d \in J}d) \prod_{\substack{d \in J\\
d \nmid s}}\frac{d-2}{d-1}\\
=&\ \phi(\prod_{d \in J} d) \prod_{\substack{d \in J\\
d \nmid s }} \left(1-\frac{1}{d-1}\right)
\end{align} 
which completes the proof.
\end{proof}

%
	%

\begin{lemma} \label{-BM-}
Let $s$ be even. Let $J$ be as in Lemma \ref{-BN-} with the additional conditions that $2 \in J$ and more than one element of $J$ divides $s$. Let $t$ be any integer for which two is the sole element of $J$ that divides $t$. Then 
 \begin{align}
 \notag  &  |\{1 \le m \le \prod_{d \in J} d: (m(m-s), \prod_{d \in J}d)=1\}|\\
 \notag  &\ \  \ \ \ \ \ \ \  \ \ \ \ \ \ \  \ \ \ \  =\ \prod_{d \in J} d \prod_{d \in J} \left(1-\frac{1}{d}\right) \prod_{\substack{d \in J \setminus \{2\}\\ d \mid s}} \left(1-\frac{1}{d-1}\right)\\
  \notag &\ \  \ \ \ \ \ \ \  \ \ \ \ \ \ \  \ \ \ \ >\ |\{1 \le m \le \prod_{d \in J} d: (m(m-t), \prod_{d \in J}d)=1\}|\\
 &\ \  \ \ \ \ \ \ \  \ \ \ \ \ \ \  \ \ \ \  =\ \prod_{d \in J} d \prod_{d \in J} \left(1-\frac{1}{d}\right) \prod_{d \in J \setminus \{2\}} \left(1-\frac{1}{d-1}\right).
\label{-GS}
\end{align}
\end{lemma}

\begin{proof}
We have $\{d \in J: d \nmid s\} \subset \{d \in J: d \nmid t\}$. Therefore
\begin{align}
	\prod_{\substack{d \in J\\
	d \nmid s }} \left(1-\frac{1}{d-1}\right) >\ \prod_{\substack{d \in J\\
	d \nmid t }} \left(1-\frac{1}{d-1}\right). \label{-DJD-}
\end{align}
We note that the right side of the second relation of \eqref{-GS} is equal to $\phi(\prod_{d \in J} d) \prod_{\substack{d \in J \setminus \{2\} }} (1-1/(d-1))$. Since $\prod_{d \in J}d \prod_{d \in J} (1-1/d) = \phi(\prod_{d \in J}d)$, combining \eqref{-DJD-} and Lemma \ref{-BN-} for $s$ as current therefore gives \eqref{-GS}.
\end{proof}

\subsection{Definition.} \label{D.A}
For any even $r$, any $n \ge 1$, any set $N$ of integers, any $p \in P(n)$, and any integer $k$ define
\begin{align}
V_{N,p,k} = \{h \in N:\ p \mid h-k \}.
\end{align} 
For any set $M$ of integers such that, for any distinct $p$ and $q$ in $P(n)$ and any $s$ and $s'$ in $\{0,r\}$,
$V_{M,p,s} \ne V_{M,q,s'}$, define
\begin{align}
\notag R(M,n,& r) \\
\notag =&\ \{F \subseteq \{V_{M,p,k}: p \in P(n)\ \&\ k \in \{0, r \}\}: \text{ for each } q \in P(n) \text{ we have } \\
 &\ \ \  \ \ \ 
|\{V_{M,q,k}: k \in \{0, r\}\} \cap F| = 1 \}.
\end{align} 
Hence $R(M,n,r)$ is the set of all sets, $G$, such that each $T \in G$ satisfies $T = \{m \in M: q \mid m-j\}$ for some $q \in P(n)$ and $j \in \{0,r\}$, and for each $p \in P(n)$, precisely one element of $\{\{m \in M: p \mid m-s\}: s \in \{0,r\}\}$ is in $G$. We note that $N$ satisfies all stated conditions on $M$ when $N$ is an interval with $|N| \ge 2p_n$. We note further that, for $i$, $r$ and $Z_{i,r}$, as in our Extended Introduction, $Z_{i,r} = R([1, i], \pi(\sqrt{2i}), r)$ when $[1,i]$ satisfies all stated conditions on $M$. Finally,
\begin{align}
	\bigcup (\bigcup R(M,n,r)) =\ \{m \in M: (m(m-r), \prod_{k=1}^n p_k) \ne 1\}.
\end{align} 

\begin{lemma}  \label{-RIJ-}
 For any integer $n$, any even $r$, any $j$ for which there exists $R([1, j], n, r)$, and any $T \in R([1, j], n, r)$ there exists an interval $J_T$ for which $|J_T|=j$ and
\begin{align}
	|\bigcup T| = |\{m \in J_T: (m, \prod_{k=1}^n p_k) \ne 1\}|. \label{-JTM-}
\end{align}
\end{lemma}

\begin{proof}
For any $p \in P(n)$ and any $v>0$, for some $j$-element interval $W_T$, the $v$-th lowest element of $\bigcup (\{\{1 \le m \le j : p \mid m-s\}: s \in \{0,r\}\} \cap T)$ is equal to $(v-1)p +  \min \{m \in W_T : p \mid m\} - \min W_T +1$. For any nonempty subset $B$ of $P(n)$, we may write $B = \{A_1, A_2, \ldots, A_{|B|}\}$ where, for each $1 \le m \le |B|-1$ we have $A_{m} \subset A_{m+1}$. By progressively increasing $|B|$ by increments of one from $|B|=1$ we see, through known modular arithmetic combined with the fact that $\prod_{p \in A_{m+1}}p / \prod_{p \in A_{m}} p$ is the element of $A_{m+1} \setminus A_m$, that we have the following. There exists an integer $i_T$ for which $1 \le i_T \le \prod_{k=1}^n p_k$, and for all $q \in P(n)$, $\{i_T \le m \le  j+i_T -1: q \mid m\} =\ T$.
Thus $[i_T, j+i_T-1]$ satisfies all stated conditions on $J_T$, giving \eqref{-JTM-}.
\end{proof}

\begin{lemma} \label{-CAP-}
For any $n$, let $b$ be any positive integer that has no factors in $P(n)$. 
Then for any subset $V$ of $P(n)$ 
\begin{align}
	\notag & |\{1 \le m \le \prod_{q \in P(n) \cup \{b\}} q:(m, \prod_{q \in V \cup \{b\}} q) \ne 1\}\\
	\notag &\ \ \  \ \ \ \ \ \ \ \  \ \ \ \ \cap\ \{1 \le m \le \prod_{q \in P(n) \cup \{b\}} q:   (m, \prod_{q \in P(n) \setminus V} q) \ne 1\}|\\
	%
	%
\notag	&\ \ \  \ \ \ \ \ =\ 
	\left(1-\prod_{q \in V \cup \{b\}} \left(1-\frac{1}{q}\right)\right)\\
	&\ \  \ \ \ \ \ \ \ \ \ \ \ \  \times |\{1 \le m \le \prod_{q \in P(n) \cup \{b\}} q: (m, \prod_{q \in P(n) \setminus V} q) \ne 1\}|.
	%
	\label{-MQV-}
\end{align}
\end{lemma}

\begin{proof}
We have I to III, below.
\vspace{\baselineskip}

\noindent I. For any distinct $v$ and $w$ in $[1, \prod_{q \in V \cup \{b\}} q]$ we have
  \begin{align}
 \notag \{v + m\prod_{q \in V \cup \{b\}} q:
		m \in \mathbb{N}\} \cap 
	\{w + m\prod_{q \in V \cup \{b\}} q:
	m \in \mathbb{N}\} = \emptyset.
\end{align} 
 Therefore, the set whose cardinality is left side of \eqref{-MQV-} is equal to
\begin{align}
	\notag & \{v' + m\prod_{q \in V \cup \{b\}} q :
1 \le v' \le \prod_{q \in V \cup \{b\}} q\ \&\ (v' + m\prod_{q \in V \cup \{b\}} q, \prod_{q \in V \cup \{b\}} q) \ne 1 \ \&\ m \in \mathbb{N}\}\\
 &\ \ \  \ \ \ \ \ \ \ \  \ \ \ \ \cap\ \{1 \le k \le \prod_{q \in P(n) \cup \{b\}} q:  (k, \prod_{q \in P(n) \setminus V} q) \ne 1\}. \label{-QVB-}
	\end{align}
	Here we note that, for the above set, when the bound variable $m \in \mathbb{N}$ is replaced with $m \in [1, \prod_{p \in P(n) \setminus V} p]$, the set is unchanged. \vspace{\baselineskip}

\noindent II. Let $B$ be the set of all integers that satisfy all stated conditions on $v$. Then for the set, $L$, of all $i \in B$ for which $(i, \prod_{q \in V \cup \{b\}}q) \ne 1$, we have $|L| =(1-\prod_{q \in V \cup \{b\}}(1-1/q))|B|$.
\vspace{\baselineskip}

\noindent III. For any $y \in \{v, w\}$, let
\begin{align}
	\notag S_y =&\ \{y + m\prod_{q \in V \cup \{b\}} q:	1 \le	m \le \prod_{q \in P(n) \setminus V}q\ \&\ (y + m\prod_{q \in V \cup \{b\}} q, \prod_{q \in P(n) \setminus V}q) \ne 1 \}.
\end{align}
Then $|S_v|=|S_w|$ with $S_v \cap S_w = \emptyset$.
		\vspace{\baselineskip}

\noindent  Combining all of I, II and III gives \eqref{-MQV-}.
\end{proof}

\subsection{Definition.} \label{D.U.I}

For any $n\ge 4$ and any set $N$ of integers for which $R(N,n,t)$ exists, let $u_{N, n}$ be any rational number such that, for each $T \in R(N,n,t)$,
\begin{align}
	\notag  |\bigcup T| \le\ |N|\left(1- (1-u_{N,n})\prod_{k=1}^n\left(1-\frac{1}{p_k}\right)\right).
\end{align}
It follows through Lemma \ref{-RIJ-} for $n$ as current, combined with Theorem \ref{Th1} for $I=N$ and $|I|=j$ and $J=P(n)$ that, for each $c \in \{1,2\}$, our assumptions on $u_{I,n}$ allow us that, when $j \ge 17$,
\begin{align}
u_{I,n}	=\ \frac{j\left(-\frac{1}{\log{cj}} + \prod_{k=1}^n\left(1-\frac{1}{p_k}\right)\right)  + n + \frac{5n^2}{8} }{j\prod_{k=1}^n\left(1-\frac{1}{p_k}\right)}. \label{-IJN-}
\end{align}
The term $n$ on the numerator is attributable to the fact that the first $n$ primes are not coprime to $\prod_{k=1}^n p_k$. Also, $x/\log{x}$, and thereby $j/\log{cj}$, is a lower bound for the prime counting function for all $x \ge 17$.\cite{2} 

\subsection{Remark}
In the ensuing exposition we occasionally introduce sets of a fixed cardinality. This is because of the self-explanatory nature of the written set.

\begin{lemma} \label{-XMT-}
Let $n$ be any positive integer. Let $r$ be even. Let $I$ be any interval for which there exists $R(I,n,r)$. 
Let $j$ be any element of $P(n) \setminus \{2\}$ for which $j \nmid r$. Let $X \in R(I,n,r)$. Let $s$ be the element of $\{0,r\}$ for which $\{m \in I: j \mid m-s\} \in X$. Let $s' \in \{0,r\} \setminus \{s\}$. Finally, let 
 \begin{align}
	\notag & f(X,j,n,r) =\ (X \setminus  \{m \in I: j \mid m-s\}) \cup \{m \in I: j \mid m-s'\}.
 \end{align}
 Then 
\begin{align}
	|\bigcup (X \cup f(X,j,n,r))| \le\ |I| \left(1-(1-2u_{I,n})\prod_{k=1}^n\left(1-\frac{1}{p_k}\right)\right). \label{-XXQ-}
\end{align}
\end{lemma}

\begin{proof}
First, we note that $f(X,j,n,r) \in R(I,n,r)$.
Also, 
\begin{align}
	\frac{\max\{|\bigcup Y|: Y \in R(I,n,r)\} - |I|\left(1-\prod_{k=1}^n\left(1-\frac{1}{p_k}\right)\right)}{|I|\prod_{k=1}^n\left(1-\frac{1}{p_k}\right)}, \label{-TJI-}
\end{align}
satisfies all stated conditions on $u_{I,n}$.
Here, when $1/u_{I,n} \notin P(n)$ we have
\begin{align}
|\bigcup X| \le\ |I|\left(1-\prod_{q \in P(n) \cup \{1/u_{I,n}\}} \left(1-\frac{1}{q}\right)\right).
\end{align} 
. 

Let $U_{X,j,n,r}$ be the set that satisfies either i or ii, below.
\vspace{\baselineskip}

\noindent i. When $|\bigcup X| \ge \lfloor|I|(1-\prod_{k=1}^n (1-1/p_k))\rfloor$, then $U_{X,j,n,r}$ is the set of all $\lfloor|I|(1-\prod_{k=1}^n (1-1/p_k))\rfloor$-element subsets, $M$, of $\bigcup X$ and $\bigcup f(X,j,n,r)$ for which 
\begin{align}
	& \notag |M \setminus \{m \in I: \{p \in P(n): p \mid m(m-r)\}= \{j\}\}|\\
	\notag &\   \ \ \ \ \ \ \ =\ \max\left\{|L \setminus \{m \in I:\{p \in P(n): p \mid m(m-r)\}= \{j\}\}| : \vphantom{\prod_{z}^z \left(1-\frac{z}{z}\right)} L \subseteq \bigcup K \text{ for some}  \right. \\
	&\   \ \ \ \ \ \ \  \ \ \ \ \  \left. K \in \{X,f(X,j,n,r)\}\ \&\ |L| =  \left \lfloor|I|\left(1-\prod_{k=1}^n \left(1-\frac{1}{p_k}\right)\right)\right \rfloor\right\}.
\end{align}
\vspace{\baselineskip}

\noindent ii. When $|\bigcup X| < \lfloor|I|(1-\prod_{k=1}^n (1-1/p_k))\rfloor$, then $U_{X,j,n,r} = \{\bigcup X\}$.
\vspace{\baselineskip}

%
%


We have
 \begin{align}
	\notag & \bigcup (\bigcup (X \cup f(X,j,n,r))) \setminus \bigcup X\\
	\notag & \ \ \ \  \ \ \ \ \  =\ \{m \in I: \{p \in P(n): p \mid m(m-r)\}= \{j\}\ \&\ j \mid m-s'\}. 
\end{align}
The fact that $\{m \in I: j \mid m-s\}$ is the sole element of $X$ that is not in $ f(X,j,n,r)$, while, by assumption, for each $K \in \{X, f(X,j,n,r)\}$,  
\begin{align}
	|\bigcup K| \le\  |I| \left(1-(1-u_{I,n})\prod_{k=1}^n\left(1-\frac{1}{p_k}\right)\right)
\end{align} 
 thereby implies the following. The fact that 
\begin{align}
	\notag |\bigcup (X \cup f(X,j,n,r))| =&\ |\bigcup X| + |\{m \in I: j \mid m-s'\}|\\
	&\ \ \ \  \  \ \  \ - |(\bigcup X) \cap  \{m \in I: j \mid m-s'\}|
\end{align}  
gives
\eqref{-XXQ-}. Here, the coefficient two for $2u_{I,n}$ is justified through the fact that
\begin{align}
	|\{\bigcup (\bigcup (X \cup f(X,j,n,r))) \setminus \bigcup X, \bigcup X\}| = 2.
\end{align}
This completes the proof.
%
%
%
%
%
%
\end{proof}

\begin{lemma} \label{-MAB-}
Let $n$, $V$ and $b$ all be as in Lemma \ref{-CAP-}. Let $a$ be any positive integer less than $b$. Then for any $a$-element subset $S$ of $[1, b]$,
\begin{align}
	\notag & |(\{1 \le m \le \prod_{q \in P(n) \cup \{b\}} q:  (m, \prod_{q \in V \cup \{b\}} q) \ne 1\}  \cup \{kwb : k \in \mathbb{N}\ \&\  w \in S\})\\
	\notag &\ \ \  \ \ \ \ \ \ \ \  \ \ \ \ \ \ \ \ \cap\  \{1 \le m \le \prod_{q \in P(n) \cup \{b\}} q:  (m, \prod_{q \in P(n) \setminus V} q) = 1\}|\\
	%
	%
	\notag &\ \ \  \ \ \ \ \ =\ 
	\left(1-\left(1-\frac{a}{b}\right)\right. \\
	&\  \  \ \ \ \  \ \ \ \ \ \ \left. \times \prod_{q \in V} \left(1-\frac{1}{q}\right)\right)|\{1 \le m \le \prod_{q \in P(n) \cup \{b\}} q:    (m, \prod_{q \in P(n) \setminus V} q) = 1\}|.
	%
	\label{-MIB-}
\end{align}
\end{lemma}

\begin{proof}
Combining Lemma \ref{-CAP-} and \eqref{-QVB-} for $v,w \in S$ gives
\begin{align}
	\notag & |(\{i + m\prod_{q \in V \cup \{b\}} q :
1 \le i \le \prod_{q \in V \cup \{b\}} q\ \&\ (i + m\prod_{q \in V \cup \{b\}} q, \prod_{q \in V \cup \{b\}} q) \ne 1\}\\
\notag &\ \  \ \ \  \ \ \ \ \ \  \cup \{kwb : k \in \mathbb{N}\ \&\  w \in S\})\\
	\notag &\ \ \  \ \ \ \ \ \ \ \  \ \ \ \ \cap\ \{1 \le m \le \prod_{q \in P(n) \cup \{b\}} q: (m, \prod_{q \in P(n) \setminus V} q) \ne 1\}|\\
	%
	%
	&\ \ \  \  =
	\left(1-\left(1-\frac{a}{b}\right)\prod_{q \in V} \left(1-\frac{1}{q}\right)\right)|\{1 \le m \le \prod_{q \in P(n) \cup \{b\}} q:   (m, \prod_{q \in P(n) \setminus V} q) \ne 1\}|.
	%
\end{align} 
Therefore, the fact that $\phi(b\prod_{q \in V}q) = \prod_{q \in V \cup \{b\}}q \prod_{q \in V \cup \{b\}}(1-1/q)$ implies \eqref{-MIB-}.
\end{proof}  

\begin{lemma} \label{-MMT-}
Let $n$, $I$ and $r$ be as in Lemma \ref{-XMT-}. Then 
\begin{align}
	\notag & \left|\left\{m \in I: \left(m(m-r), \prod_{k=1}^n p_k\right) = 1\right\}\right| \\
	&\ \ \  \ \ \ \ \  \  \ \ \ \  \ \ \ \ge \ |I|(1-2u_{I,n})\prod_{k=1}^n\left(1-\frac{1}{p_k}\right)\prod_{\substack{p \in P(n)\\ p \nmid r}} \left(1-\frac{1}{p-1}\right). \label{-MIM-} 
\end{align} 
\end{lemma} 

\begin{proof}
Through our forthcoming \eqref{-GTI-}, we shall use Lemmas \ref{-BM-}, \ref{-XMT-} and \ref{-MAB-} to show that 
\begin{align}
 	|\bigcup (\bigcup R(I,n,r))| \le\ |I|\left(1-(1-2u_{I,n})\prod_{k=1}^n\left(1-\frac{1}{p_k}\right)\prod_{\substack{p \in P(n)\\ p \mid r}}\left(1-\frac{1}{p-1}\right)\right)
, \label{-RIN-} 
\end{align}
from which \eqref{-MIM-} immediately follows.
We begin by noting I, below. 
\vspace{\baselineskip}

\noindent I. For any $T \in R(I,n,r)$, let $K_T$ be any subset of $T$.
Then for each $M \in \{K_T, T\}$, 
\begin{align}
	\notag |\bigcup M| \le &\ |I|\left(\left(1- (1-u_{I,n})\prod_{\substack{q \in P(n)|\ \{m \in I:\\ q \mid m-s \} \in K_T \\ \text{ for some } s \in \{0,r\}}} \left(1-\frac{1}{q}\right)\right)
	\vphantom{\left(1-u_{I,n}\right)\prod_{\substack{q \in P(n)|\ \{m \in I:\\ q \mid m-s \} \in K_T \\ \text{ for some } s \in \{0,r\}}} \left(1-\frac{1}{q}\right)\left(1-\prod_{\substack{q \in P(n)|\ \{m \in I:\\ q \mid m-s \} \in T \setminus K_T \\ \text{ for some } s \in \{0,r\}}} \left(1-\frac{1}{q}\right)\right) }    \right.\\
	\notag &\ \  \ \ \ \left. + \left(1-\prod_{\substack{q \in P(n)|\ \{m \in I:\\ q \mid m-s \} \in T \setminus K_T \\ \text{ for some } s \in \{0,r\}}} \left(1-\frac{1}{q}\right)\right)\left(1-u_{I,n}\right)\prod_{\substack{q \in P(n)|\ \{m \in I:\\ q \mid m-s \} \in K_T \\ \text{ for some } s \in \{0,r\}}} \left(1-\frac{1}{q}\right)\right)\\ 
	=& \ |I|\left(1- (1-u_{I,n})\prod_{k=1}^n\left(1-\frac{1}{p_k}\right)\right). \label{-MIU-}
\end{align}
We see here that, on specifying $T$, the right side of the first relation is a constant function of $K_T$. 
\vspace{\baselineskip}

\noindent Let $d$ be any integer that has no factors in $P(n)$ and for which there exists an integer $d'$ less than $2d$ for which $u_{I,n} = d'/d$. Each term between the outer brackets on the right side of the first relation of \eqref{-MIU-} is found through Lemma \ref{-MAB-} for 
\begin{align}
	\notag V=\ \{q \in P(n): \{m \in I: q \mid m-s \} \in K_T \text{ for some } s \in \{0,r\}\}
\end{align}
and $a = d'$ and $b=d$, combined with the fact that, for each $H \in \{ V \cup \{b\}, P(n) \setminus V\}$, we have $\phi(\prod_{q \in H}q)/\prod_{q \in H}q = \prod_{q \in H}(1-1/q)$.

Denote 
\begin{align}
\notag 	g_{K_T}=\  \left(1-\prod_{\substack{q \in P(n)|\ \{m \in I:\\ q \mid m-s \} \in T \setminus K_T \\ \text{ for some } s \in \{0,r\}}} \left(1-\frac{1}{q}\right)\right)\left(1-u_{I,n}\right)\prod_{\substack{q \in P(n)|\ \{m \in I:\\ q \mid m-s \} \in K_T \\ \text{ for some } s \in \{0,r\}}} \left(1-\frac{1}{q}\right),
\end{align} 
whence $g_{K_T}$ is the second term between the outer brackets on the right side of the first relation of \eqref{-MIU-}. Then, for each $p \in P(n)$,  
\begin{align}
	\notag g_{\substack{ T \setminus \{\{m \in I: p \mid m-s\}:\\ s \in \{0,r\}\}}} =&\  \frac{1-u_{I,n}}{p}\prod_{q \in P(n) \setminus \{p\}} \left(1-\frac{1}{q}\right). \label{-GTI-}
\end{align}
Further, $g_T = 0$. 
%
%

 When $p \nmid r$ we have $\{m \in I: p \mid m\} \cap \{m \in I: p \mid m-r\} = \emptyset$. We may show now that combining Lemma \ref{-XMT-} for $j=p$ and I (above) gives \eqref{-RIN-}. Here, for all $M \in R(I,n,r)$, the set $H = \{\{m \in I: p \mid m-s\}: s \in \{0,r\}\}$ is not a subset of $M$, since $p \nmid r$. However, some element of $H$ is in $M$. Hence, for our above reference to Lemma \ref{-XMT-}, our assumptions on $u_{I,n}$ allow us to assume that for some $Y \in R(I,n,r)$
\begin{align}
\notag u_{I,n} \ge\ 
	\frac{\|\bigcup Y| - |I|\left(1-\prod_{k=1}^n\left(1-\frac{1}{p_k}\right)\right)}{|I|\prod_{k=1}^n\left(1-\frac{1}{p_k}\right)}
\end{align} 
and we choose $X$ and $j$ as in Lemma \ref{-XMT-} so that $|Y| = \max\{|\bigcup X|, |\bigcup f(X,p,n,r)|\}$ where $f$ is as in Lemma \ref{-XMT-}. Here we recall that $f(X,j,n,r) \in R(I,n,r)$.
Our assumptions on $u_{I,n}$ allow us, also, to assume that $u_{I,n} = 0$ when, for all $M \in R(I,n,r)$, we have $|\bigcup M| \le |I|(1-\prod_{k=1}^n(1-1/p_k))$. Thus combining Lemma \ref{-MAB-} (for $a=2d'$ and $b=d$ and $V =  P(n) \setminus \{p\}$ and subsequently $V =\{p\}$) and \eqref{-MIU-} for each $T \in \{X, f(X,p,n,r)\}$, itself combined with the fact that, as already noted, $g_T =0$, and finally with Lemma \ref{-BM-} for $J=P(n) \cup \{d\}$ (whence we find the products over $k$ in \eqref{-RIN-}), gives \eqref{-RIN-}. With respect to Lemma \ref{-BM-}, we note here that, for any real $x$, 
\begin{align}
	\notag \left(1-\frac{1}{x}\right)\left(1-\frac{1}{x-1}\right) =&\ \frac{(x-1) \frac{x-2}{x-1}}{x}\\
\notag 	=&\ 1-\frac{2}{x}.
\end{align}
We substitute first $x = d$ and subsequently, when combining \ref{-MAB-} with Lemma \ref{-BM-}, $x=d'/d$, and finally $a= 2d'$ and $b=d$ to reach \eqref{-RIN-}. 

Since 
\begin{align}
\bigcup (\bigcup R(I,n,r)) =\ \{m \in I: (m(m-r), \prod_{k=1}^n p_k) \ne 1\},
\end{align} 
we may replace $(m(m-r), \prod_{k=1}^n p_k) \ne 1$ with $(m(m-r), \prod_{k=1}^n p_k) = 1$ whence \eqref{-RIN-} implies \eqref{-MIM-}.
\end{proof}

\subsection{Remark}\label{-AS-} 
 Let $n>0$. Since, for all $m$ such that $0 < m < p_{n}^{2}$ and all $k \ge 1$, we have $\{p \in P(n): p \mid m\} \ne \{p_{n+k}\}$, there is no composite in $\{1 \le j \le p_n^2: (j, \prod_{k=1}^n{p_k}) = 1\}$ (hence our respecting the sieve of Eratosthenes, by employing $p_n^2/2$ in our specifying the cardinality of intervals with which we are working). Therefore, for all integers, $z_n$, such that $p_n^2/2< z_n < p_{n+1}/2$, for $s(m)=1$ and $s'(m) = m-2z_n$, each $y \in \{s, s'\}$ and each $d \in \{1,2\}$,
\begin{align}
 \notag \left\{1 \le m \le  dz_n: \left(m y(m), \prod_{k=1}^n{p_k}\right) = 1\right\}
\end{align} 
is a subset of the set of all the primes less than or equal to $dz_n$.


\subsection{Definition.} \label{-WPV-} 

For any $x>1$, let 
\begin{align}
\notag \text{Hi}(x)=\ \frac{x}{\log{x}} \left(1 + \frac{1}{\log{x}}+ \frac{2.51}{\log^{2}{x}}\right).	
\end{align}
It is a result of Dusart \cite{1} that $\pi(x) < \text{Hi}(x)$ \label{-D2}
for all $x \geq 355,991$. We find
$\pi(355991) = 30,456$.
Also, for $x \ge 355,991$, $\text{Hi}(x)\log(x)/x$ is strictly decreasing to one.

\subsection{Definitions.} \label{D.QK}
For any real $r$, let $q_r$ be the highest $x$ such that $\text{Hi}(x) = \text{Hi}(r) + 1$.

We have
\begin{align}
	\notag \text{Hi}(355991) =&\  \frac{355991}{\log(355991)}\left(1+\frac{1}{\ln(355991)}+\frac{2.51}{\log^2(355991)}\right) \\
	\approx &\ 30456.026
\end{align} 
and 
\begin{align}
\notag  	q_{355991} \approx&\ 356003.80\\
 		q_{356003.80} \approx&\ 356016.58. \label{-QD-}
\end{align} 
It follows through the fact that $\log{r}$ is increasing that $q_r-q_{r-1}$ is an increasing function of $r$

For any positive integer $k$, let $v(k)$ be the real number such that  $(v(1), v(2), v(3), \ldots )$ is the sequence of real numbers for which  I to III, below, are all true:
\vspace{\baselineskip}

\noindent I.  $\{j : j = v(u) \text{ for some } 1 \le u \le 30456\} = P(30456)$;
\vspace{\baselineskip}

\noindent II. for all $c \ge 30458$ we have $v(c) = q_{v(c-1)}$;
\vspace{\baselineskip}

\noindent III. for $t$ such that $\text{Hi}(t) = 30456$, we have $v(30457) = q_{t}$, whence $t \approx 355990.667$. 

\noindent Then through the previously cited result of Dusart, $v(c) < p_c$. Also, $v(c+1)-v(c)$ is strictly increasing. 
We note that $p_{30456} = 355969 < q_{p_{30455}} \approx 355979.783 $. Contrastingly $p_{30457} = 356023 > 356003.456 \approx q_{355990.667} \approx v(30457)$. 

\subsection{Remark} \label{-RBU-}
Let $n$ be any integer. In the ensuing lemma, the use of $2((n+1)^2-n^2)$ as a denominator is designed to invoke, in a more congenial expression, the $5n^2/8$
that is found through Theorem \ref{Th1} for $I=[1+i, \lfloor p_n^2 \rfloor + i]$ for some integer $i$ and $J=P(n)$. 

\begin{lemma} \label{-OLQ-}
Let $n \ge 30457$. 
Then
\begin{align}
 	\notag \frac{\frac{v(n+1)^2}{\log(v(n+1)^2)}-\frac{v(n)^2}{\log{v(n)^2}}}{2((n+1)^2-n^2)}   =\ \log{v(n)} + O\left(\frac{\log^2{v(n)}}{v(n)}\right).
	\end{align} 
\end{lemma}

\begin{proof} 
We have
\begin{align}
		\notag  \frac{\frac{v(n+1)^2}{\log(v(n+1)^2)}-\frac{v(n)^2}{\log{v(n)^2}}}{2((n+1)^2-n^2)} 
	        =&\ \frac{\frac{v(n+1)^2}{\log(v(n+1)^2)}-\frac{v(n)^2}{\log{v(n)^2}}}{4n + 2}\\ 
	\notag \\
\notag  	=&\ \frac{\frac{v(n+1)^2}{\log(v(n+1)^2)}-\frac{v(n)^2}{\log{v(n)^2}}}{\frac{v(n)}{\log{v(n)}}+ O\left(\frac{v(n)}{\log{v(n)}}\right)}\\
\notag \\
\notag  	=&\ \frac{v(n+1)^2 -v(n)^2}{v(n) + O\left(\frac{v(n)}{\log{v(n)}}\right)}
\\
\notag \\
\notag  	=&\ \frac{2v(n)\log(v(n)) + \log^2{v(n)}}{v(n) + O\left(\frac{v(n)}{\log{v(n)}}\right)}\\
\notag \\ 
\notag  	= &\ \log{v(n)}  + \frac{\log^2{v(n)}}{v(n)} + O\left(\frac{\log^2{v(n)}}{v(n)}\right)\\
\notag \\ 
\notag  	= &\ \log{v(n)} + O\left(\frac{\log^2{v(n)^2}}{v(n)}\right). 
\end{align}
The second relation follows through the Prime Number theorem, whereby for real $x$, $\pi(x) \sim x/\log{x}$. In the fourth relation, the fact that $\text{Hi}(x)\log(x)/x$ is decreasing to one implies that $v(n+1) - v(n) \sim \log{v(n)}$. More precisely, the fact that $\text{Hi}(x)\log(x)/x$ is strictly decreasing to one implies that $(m_x-x)/(q_x-x)$ is strictly decreasing to one, where $m_x$ is the highest $j$ such that $j/\log{j} = 1+x/\log{x}$. 
Thus we substitute
\begin{align}
	(v(n)+\log{v(n)})^2-v(n)^2 = 2v(n)\log(v(n)) + \log^2{v(n)}
\end{align} 
 for $v(n+1)^2 - v(n)^2$. This completes the proof.
\end{proof}

\subsection{Proof of Theorem \ref{Th2}}
\begin{proof}
In our proof we tacitly use $i$ and $r$ as in our Introduction to Theorems \ref{Th2} and \ref{Th3}, with $i=\pi(\sqrt{2i})$ and $r=2$. More precisely, for any integer $n$, we use Lemma \ref{-MMT-} in conjunction with the fact that, for any $m \le p_n^2$, when $m$ and $m-2$, and thereby $m(m-2)$, are each coprime to $\prod_{k=1}^n p_k$, we see that $m$ and $m-2$ are together a prime pair.

The fact that, for any $m$, $(m(m-2), \prod_{k=1}^n p_k) =1$ if and only if both $(m, \prod_{k=1}^n p_k) =1$ and $(m-2, \prod_{k=1}^n p_k) =1$, implies through Remark \ref{-AS-} first that
\begin{align}
\notag  & \left\{(m, m-2): 1< m < p_n^2\ \&\ \left(m(m-2), \prod_{k=1}^{n} p_k \right) = 1 \right\}\\
 & \ \ \ \   =\ \left\{(p, p-2) : p_n +2 < p < p_n^2\ \&\ p \text{ prime } \ \&\ p-2 \text{ prime}\right\}
\end{align}
and, thereby, second that
\begin{align}
\notag & \left|\left\{1 < m \le p_n^2: \left(m(m-2), \prod_{k=2}^{n} p_k \right) = 1 \right\}\right|\\
 & \ \ \ \   =\ \left|\left\{(p, p-2) : p_n +2 < p \le p_n^2\ \&\ p \text{ prime } \ \&\ p-2 \text{ prime}\right\}\right|.	\label{-GYG-}
\end{align} 


The Mertens theorem \cite{3} is given by 
\begin{align}
	\notag \lim_{n \to \infty} \log(p_n) \prod_{k=1}^n\left(1-\frac{1}{p_k}\right) =\ e^{-\gamma},
\end{align}
where $e$ is the Euler number and $\gamma$ is the Euler-Mascheroni constant. Since $\log{p_n^2} = 2 \log{p_n}$, it follows through the Prime Number theorem that
\begin{align}
	\notag \lim_{n \to \infty} \left(\sum_{k=1}^{\pi(p_n^2)} \frac{(\log(p_k)-\log{p_{k-1}})}{\pi(p_n^2)}\right) \prod_{k=1}^n\left(1-\frac{1}{p_k}\right) =&\ 2 e^{-\gamma}\\
	\approx &\ 1.12292,
\end{align}
approximating upwards, so we may impose the assumption on $u_{[1, p_n^2],n}$  that $u_{[1, p_n^2],n} \sim 1-1/1.12292 < 1/2$. Combining Lemma \ref{-OLQ-}, noting that $v(n) < p_n$ for $n \ge 30457$, and the Mertens theorem, all combined in turn with Lemma \ref{-MMT-} for $I=[1+i, p_n^2+i]$, where $i$ is any integer, themselves combined with \eqref{-IJN-} for $j=p_n^2$, $c=2$ and $n$ as current, thereby gives
\begin{align}
|\{(p, p-2) : p_n +2 < p \le p_n^2\ \&\ p \text{ prime } \ \&\ p-2 \text{ prime}\}| \to \infty. \label{-PPN-}
\end{align}

Since the set whose cardinality is on the left side of \eqref{-PPN-} is a subset of all pairs, $(p, p-2)$, such that $p$ is prime and less than $p_n^2$ and for which $p-2$ is also prime, through Remark \ref{-AS-} for $d=2$, the proof is complete.
\end{proof}


 \begin{theorem}\label{Th3} 

Let $n>4$. Then for any integer $z_n$ for which $p_n^2/2 < z_n < p_{n+1}^2/2$ and any $s_n$ for which
\begin{align}
	\notag s_n \ge\ -\frac{\pi(z_n)}{z_n} + \prod_{k=1}^{n}\left(1-\frac{1}{p_k}\right)
\end{align}
and
\begin{align}
	z_n-\left(1-\frac{2\left(z_n s_n + n +\frac{5 n^2}{8}\right)}{z_n\prod_{k=1}^n \left(1-\frac{1}{p_k}\right)}\right)i\prod_{k=1}^n\left(1-\frac{1}{p_k}\right)\prod_{k=1}^n\left(1-\frac{1}{p_k-1}\right) \ge\ 1, \label{-SNN-}
\end{align}
there exist primes $p$ and $q$ that satisfy the Goldbach equation $p+q = 2z_n$.
\end{theorem}

\begin{proof}
We introduce our proof with I, below.
\vspace{\baselineskip}

\noindent I. Suppose that $w$ is an integer greater than $p_{30456}^2/2$. Then the number of ways of writing $w$ as the arithmetic mean of two primes is greater than or equal to the cardinality of
\begin{align}
\notag L = \left\{1 \le m \le w: 
\left(m(2w-m), \prod_{k=1}^{\pi(\sqrt{2w})} p_k\right) = 1 \right\}. 
\end{align} 
Here, for any two positive integers $p$ and $q$ for which $p<q$ and $w$ is the arithmetic mean of $p$ and $q$, and $pq$ is coprime to $\prod_{k=1}^{\pi(\sqrt{2w})}{p_k}$, using the sieve of Eratosthenes we see that $p$ and $q$ are both prime; also, $p$ is in $L$ with $2w-p=q$; and for any two primes, $a$ and $b$, the average of which is $w$, $a+ b =2w$ satisfies the Goldbach equation. We note here that $p(2w-p)$ is coprime to $\prod_{k=1}^n p_k$ if and only if both $p$ and $2w-p$ are coprime to $\prod_{k=1}^n p_k$ and we substitute $p=m$ and $q=2w-p$ where $m$ is the bound variable used for $L$.
\vspace{\baselineskip}

Substituting $w=z_n$, Theorem \ref{Th3} follows by I combined with the fact that, for any even $r$,  
\begin{align}
	\notag & \left|\left\{1 \le m \le z_n : \left(m(m-r), \prod_{k=1}^n p_k\right) = 1 \right\}\right|\\
 &\ \  \ \ \ \ \ \ \ \  \ \ \  \  \ge\  z_n-\frac{2\left(z_n s_n +n +\frac{5 n^2}{8}\right)}{\prod_{k=1}^n \left(1-\frac{1}{p_k}\right)} \prod_{k=1}^n \left(1-\frac{1}{p_k}\right)\prod_{k=2}^n \left(1-\frac{1}{p_k-1}\right). \label{-ZNZ-}
\end{align}
The right side of \eqref{-ZNZ-} is found by Lemma \ref{-RIJ-} for $j=z_n$ combined with Lemma \ref{-MMT-} for $I=[1+c, z_n+c]$ where $c$ is any integer, by which we have \eqref{-IJN-}, and we substitute 
\begin{align}
 	u_{I,n} =\ \frac{z_n s_n + n +\frac{5 n^2}{8}}{z_n\prod_{k=1}^n \left(1-\frac{1}{p_n}\right)}.
\end{align} 
Here, the $5n^2/8$ term on the numerator is found through Theorem \ref{Th1} for $I$ as current and $J=P(n)$. The $n$ term is found by the fact that the first $n$ primes are not coprime to $\prod_{k=1}^n p_k$ combined with the fact that $\pi(v(n)^2/2) \le v(n)^2/(2\log{v(n)^2/2}) < z_n/\log{v(n)^2}$. The products over $k$ in \eqref{-ZNZ-} are found through Lemma \ref{-BM-}, specifically for the inequality in \eqref{-GS}, combined with Lemma \ref{-MMT-} for $r=2z_n$, wherein we may take it that two is the sole element of $P(n)$ that divides $r$. We now have II, below.
\vspace{\baselineskip}

\noindent II. The sieve of Eratosthenes justifies, through Remark \ref{-AS-} for $d=1$, our assumption in I (above) that the bound variable $m$ appearing on the left side of \eqref{-ZNZ-} may be taken to be equal to $p$, with $q=2w-p$ and $r=2w$. Thus
\begin{align}
\notag  & \left\{(m, 2z_n-m): 1 < m < z_n^2\ \&\ \left(m(2z_n-m), \prod_{k=1}^{n} p_k \right) = 1 \right\}\\
 & \ \ \ \   =\ \left\{(p, q) : p_n < p \le z_n \le q < 2z_n \ \&\ p \text{ prime } \ \&\ q \text{ prime }\ \&\ p+q=2z_n\right\}. \label{-MMR-}
\end{align}
Also
\begin{align}
\notag & \left|\left\{1 \le m \le z_n: \left(m(m-2z_n), \prod_{k=2}^{n} p_k \right) = 1 \right\}\right|\\
 & \ \ \ \ \ \  \ \ \   =\ \left| \left\{(p, q) : p_n < p \le z_n < q < 2z_n \ \&\ p \text{ prime } \ \&\ q \text{ prime } \ \&\ p+q=2z_n \right\}\right|.	\label{-GYJ-}
\end{align} 
In \eqref{-MMR-}, for the bound variable $m$, we have changed $(m(m-2z_n), \prod_{k=1}^n p_k) =1$, as in the preceding exposition for $r=2z_n$, to $(m(2z_j-m), \prod_{k=1}^j p_k) =1$. The set whose cardinality is the right side of \eqref{-GYJ-} is a subset of the set of all pairs, $(p, q)$, of primes such that $p_n < p \le z_n$ for which $z_j = (p+q)/2$. 
\vspace{\baselineskip}

Combining \eqref{-ZNZ-} and \eqref{-GYJ-} gives Theorem \ref{Th3}.
\end{proof} 

  \begin{theorem}\label{Th4} 
If the Riemann hypothesis is true, the Goldbach conjecture is true.
\end{theorem}

\subsection{Definition.} \label{D.TH}
For any $x$, let $\theta(x) = \sum_{j=1}^{\pi(x)} \log{p_j}$. 

\begin{lemma} \label{-PMT-}
For all $s>30457$,
\begin{align}
	\frac{\log{p_s}}{\log{\theta{p_m}}} <\ 1.007662 \label{-TWL-}
\end{align}
\end{lemma}

\begin{proof}
For any positive integer $k$, let $j(k)$ be the highest $y$ such that $y/\log{y} = k$. Then $j(s) > p_s$. Recall that the lowest $k$ for which $v(k) \ne p_k$ is $30457$, and that $v(s) < p_s$. We have I, below.
\vspace{\baselineskip}

\noindent I. For all $n \ge 30457$, the ratios $v(n)/j(n)$, $(v(n+1)-v(n))/(j(n+1)-j(n))$ and $\log(j(n+1))/(j(n+1)-j(n))$ are all strictly increasing to one. 
\vspace{\baselineskip}

We have
\begin{align}
	\notag 1.084175 \approx &\ \frac{j(30458)-j(30457)}{\log{v(30458)}}\\
	\notag <&\   \frac{j(30457)}{\theta(p_{30456})+\log{v(30457)}}\\
	\approx &\ 1.102878. \label{-TTL-}
\end{align} 
Since, for each $t \in \{j,v\}$, $t(s) = \sum_{k=1}^s (t(k)-t(k-1))$, combining I and \eqref{-TTL-} gives
\begin{align}
	\frac{p_s}{\theta(p_s)}  <&\ \frac{p_s}{\sum_{k=1}^s \log{v(k)}}\\
	<&\ \frac{j(s)}{\sum_{k=1}^s \log{v(k)}} \\
	<&\   \frac{j(30457)}{\theta(p_{30456})+\log{v(30457)}}.
\end{align}
Here, $j(30457) \approx 392277.800878$, $j(30458) \approx 392291.764798$ and $v(30457) \approx 356003.455995$ and $\theta(p_{30456}) \approx 355685.674752$. Therefore,
\begin{align}
	\notag \frac{\log{p_s}}{\log{\theta(p_m)}}  	<&\   \frac{\log{j(30457)}}{\log(\theta(p_{30456})+\log{v(30457)})}\\
	\approx &\ 1.007662 
\end{align}
approximating upwards, whence we have \eqref{-TWL-}.
\end{proof}

\subsection{Proof of Theorem \ref{Th4}}

 \begin{proof}
Our proof may take I in the proof of Theorem \ref{Th3} for its introduction, with the following added. It is a result of Nicolas \cite{4} that if, for all $k \ge 2$,
\begin{align}
	\frac{N_k}{\phi(N_k) \log \log N_k} >\ e^{\gamma} \label{-GBI-}
\end{align}
where $N_k = \prod_{j=1}^k p_j$ and $\gamma$ is the Euler-mascheroni constant, the Riemann hypothesis is true.
Therefore, if the Riemann hypothesis is true, we have, for all $n$,
\begin{align}
\notag 	\frac{1}{\log\left(\log{\prod_{j=1}^n p_j} \right)\prod_{j=1}^n \left(1-\frac{1}{p_j}\right)} =&\  
	\frac{1}{\log\left(\sum_{j=1}^n \log{p_j} \right) \prod_{j=1}^n \left(1-\frac{1}{p_j}\right)}\\
	>&\ e^{\gamma}. \label{-LJL-}
\end{align}
Let $t_x$ be any real number for which for all $y>x$ we have $t_x > \log(\sqrt{y})/\log{\theta(\sqrt{y})}$. Then since (as mentioned earlier), for all $x \ge 17$, $\pi(x) < x/\log{x}$, it follows by \eqref{-LJL-} that
\begin{align}
	\pi(y) >\  \frac{y e^{\gamma}}{2t_x}\prod_{j=1}^{\pi(y)} \left(1-\frac{1}{p_j}\right). \label{-PDX-}
\end{align}
Here, the coefficient 'two' in the denominator of the first quotient, is found by the fact that $\log{y} = 2 \log{\sqrt{y}}$. Hence \eqref{-PDX-} is found by a known lower bound on the prime count in the way that, through the Mertens theorem combined with the Prime Number theorem, whereby $\theta(x) \sim x$ and $\pi(x) \sim x/\log{x}$, we also have 
\begin{align}
\notag 	\lim_{x \to \infty} \frac{\pi(x)}{\frac{xe^{\gamma}}{2}\prod_{j=1}^{\pi(x)} \left(1-\frac{1}{p_j}\right)} =\ 1.
\end{align}
For $n$ such that $p_n \ge x$, the fact that $v(n) \le p_n$ thereby implies that
\begin{align}
	\pi(v(n)^2) >\  \frac{ v(n)^2 e^{\gamma}}{2t_x}\prod_{j=1}^{n} \left(1-\frac{1}{v(j)}\right). 
\end{align} 
Using, for convenience, $v(n)^2/(2\log{v(n)^2})$ as a lower bound on $\pi(v(n)^2/2)$, we thereby have
\begin{align}
	\frac{v(n)^2\prod_{j=1}^{n} \left(1-\frac{1}{v(j)}\right)}{2}-\pi\left(\frac{v(n)^2}{2}\right) <\  \frac{1}{2} v(n)^2\left(1-\frac{ e^{\gamma}}{2t_x}\right)\prod_{j=1}^{n} \left(1-\frac{1}{v(j)}\right)-n \label{-JJN-}
\end{align}
where the $-n$ term on the right side is found by the fact that the first $n$ primes are not coprime to $\prod_{j=1}^n p_j$. Through Lemma \ref{-PMT-}, we now substitute $t_x = 1.007662$, with $x=v(30457)^2$, thence to combine \eqref{-JJN-} and Theorem \ref{Th1} for $I=[1+i, \lfloor v(n)^2/2\rfloor+i]$ where $i$ is any integer, and $J=P(n)$ to give the following. First, for Lemma \ref{-MMT-}, we may now substitute 
\begin{align}
	u_{[1, \lfloor v(n)^2/2\rfloor], n} =\ \frac{\frac{1}{2} v(n)^2\left(1-\frac{ e^{\gamma}}{2\times 1.007662}\right) + \frac{5n^2}{8}}{\frac{1}{2} v(n)^2\prod_{j=1}^{30457} \left(1-\frac{1}{v(j)}\right)}. 
\end{align} 
Here, the $-n$ term that appears in \eqref{-JJN-} becomes superfluous to consideration, since $u_{[1, \lfloor v(n)^2/2\rfloor], n}$ is an upper bound on the number of integers that are not coprime to $\prod_{j=1}^n p_k$. The $5n^2/8$ is found through Theorem \ref{Th1}.
 Second, on the above substitution we have, for $r$ as in Lemma \ref{-MMT-},
\begin{align}
\notag & \frac{v(30457)^2}{2}\left(1-2	u_{[1, \lfloor v(30457)^2/2\rfloor], n}\right)\prod_{j=1}^{30457}\left(1-\frac{1}{v(j)} \right)\prod_{\substack{1 \le j \le n\\ p_j \nmid r}}^{30457}\left(1-\frac{1}{v(j)-1} \right)\\
\notag & \ \  \ \ \    \ge \ \frac{v(30457)^2}{2}\left(1-2	u_{[1, \lfloor v(30457)^2/2\rfloor], n}\right)\prod_{j=1}^{30457}\left(1-\frac{1}{v(j)} \right)\prod_{j=2}^{30457}\left(1-\frac{1}{v(j)-1} \right) \\
& \ \  \ \ \   \approx   56,611,211.95. \label{-VJI-}
\end{align}
The first relation follows through Lemma \ref{-BM-}, specifically for the inequality in \eqref{-GS}, for $P(30457)=J$. Then, for $z_n$ as in the Proof of Theorem \ref{Th3} for $n$ as current, the facts that we may increase $z_n$ by increments of one and that $z_n > p_n > v(n)$ and that $(1-1/v(n))(1-1/(v(n)-1))$ is increasing, all combined with Lemma \ref{-OLQ-}, implies the following. Combining \eqref{-VJI-} with Theorem \ref{Th3} for $s_n = 1-e^{\gamma}/(2\times 1.007662) - n/z_n$, gives I, below, for $i=2z_n$.
\vspace{\baselineskip}

\noindent I. If the Riemann hypothesis is true, then all even numbers, $i$, such that $\pi(\sqrt{i}) \ge 30457$ are the sum of two primes. 
\vspace{\baselineskip}

All even numbers, up to values greater than $p_{30457}^2= 126,752,376,529$, which is less than $10^{12}$, have been shown to be the sum of two primes by, for example, Richstein \cite{5}, who verified the Goldbach conjecture for all even numbers up to $4 \times 10^{14}$. Since Theorem \ref{Th4} is, subject to the Riemann hypothesis being true, a statement of the Goldbach conjecture, it follows by I that if the Riemann hypothesis is true, the Goldbach conjecture is also true, which completes the proof.
\end{proof}

\subsection{Conclusion}
We have shown through a single basic method that the Twin Primes is true, and the Goldbach conjecture is true subject to the Riemann hypothesis being true.
%
%

%


\vspace{\baselineskip}




\begin{thebibliography}{2}

\bibitem{1} P. Dusart. Autour de la Fonction Qui Compte le Nombre de Nombres Premiers. Doctoral thesis for l'Universit\'e\ de Limoges, 1998. 


\bibitem{2} J. B. Rosser and L. Schoenfeld, Approximate formulas for some functions of prime numbers, Ill. J. Math., 6, 1962, pp.~64--94. 

\bibitem{3}
 J. Havil. Exploring Euler's Constant. Princeton, NJ: Princeton University Press, 2003.


\bibitem{4} J. Nicolas, Petites valeurs de la fonction d'Euler, J. Number Theory 17, 3, 1983, pp.~375--388.


\bibitem{5} Richstein, J. Verifying the Goldbach Conjecture up to $4\times10^{14}$. Math. Comput. 70.
2001.

%

\end{thebibliography}
 \end{document}